\documentclass{article}

\usepackage[english]{babel}
\usepackage[utf8x]{inputenc}
\usepackage[T1]{fontenc}

\usepackage[letterpaper,margin=1.5in,marginparwidth=1in]{geometry}

\usepackage{amsmath}
\usepackage{amssymb}
\usepackage{amsthm}
\usepackage{graphicx}
\usepackage{enumerate}
\usepackage{tikz}
\usepackage[colorinlistoftodos]{todonotes}
\usepackage[colorlinks=true, allcolors=blue]{hyperref}

\title{Mahavier Products, Idempotent Relations, and Condition $\Gamma$}
\author{Steven Clontz and Scott Varagona}

\newcommand{\term}{\textit}
\newcommand{\vect}{\mathbf}
\newcommand{\tuple}[1]{\left\langle{#1}\right\rangle}
\newcommand{\genInvLim}[1]{\varprojlim\tuple{#1}}
\newcommand{\maProd}[1]{{\mathbf{M}}\tuple{#1}}
\newcommand{\rest}{\upharpoonright}

      \theoremstyle{plain}
      \newtheorem{theorem}{Theorem}
      \newtheorem{lemma}[theorem]{Lemma}
      \newtheorem{corollary}[theorem]{Corollary}
      \newtheorem{proposition}[theorem]{Proposition}

      \newtheorem{observation}[theorem]{Observation}

      \theoremstyle{definition}
      \newtheorem{definition}[theorem]{Definition}

      \theoremstyle{remark}

\begin{document}
\maketitle

\begin{abstract}
Clearly, a generalized inverse limit of metrizable spaces 
indexed by \(\mathbb N\) is metrizable,
as it is a subspace of a countable product of metrizable spaces.
The authors previously showed that all idempotent, upper semi-continuous, surjective,
continuum-valued bonding functions on \([0,1]\) (besides the identity) satisfy a certain Condition \(\Gamma\);
it follows that only in trivial cases can a generalized inverse limit of
copies of \([0,1]\) indexed by an uncountable ordinal be metrizable. 
The authors show that Condition \(\Gamma\)
is in fact guaranteed by much weaker criteria, proving a more general metrizability theorem for certain Mahavier Products.
\end{abstract}

\section{Introduction}

In the spirit of the celebrated work by Mahavier \cite{Mahavier} and Ingram \& Mahavier \cite{IngMah1}, 
who first generalized traditional inverse limits to those with set-valued bonding 
functions, researchers have sought more ways to generalize inverse limits. One 
route proposed in \cite{IngMahbook} would be to index the factor spaces of an inverse limit not 
necessarily by the natural numbers, but rather, by some other directed set. 
However, a poorly-behaved directed set can cause an inverse limit to be empty \cite{IngMahbook};
therefore, most research in this area has involved inverse limits with totally  
ordered index sets.

Various results in recent years have shown that investigating totally 
ordered index sets besides the natural numbers is fertile ground for new work. 
Ingram and Mahavier laid the foundation in \cite{IngMahbook} by proving some basic connectedness 
theorems. Notably, Patrick Vernon in \cite{vernon} studied inverse limits on $[0,1]$ with 
set-valued functions indexed by the integers, and he showed that such an inverse 
limit with a single bonding function could be homeomorphic to a 2-cell---a 
striking result, considering Van Nall had shown in \cite{vannall} that this could never 
happen for an inverse limit indexed by the natural numbers.
Later, the notion of a generalized inverse limit was further generalized to the notion of a Mahavier product, which only required the index set to be a preordered set. Thus, the recent interest in Mahavier products (e.g., Greenwood \& Kennedy \cite{GreenKen} and Charatonik \& Roe \cite{charroe}) has helped bring ``alternate'' index sets further into the mainstream.
In particular, in \cite{charroe}, Charatonik and Roe proved theorems about Mahavier products indexed by arbitrary totally ordered sets, and in the process introduced some helpful terminology in the study of generalized inverse limits. Therefore, given this context, it is natural to consider generalized inverse limits (or Mahavier products) indexed by ordinals.

In \cite{CLONTZVARAGONA} the authors studied the special case of a generalized inverse limit 
of copies of \([0,1]\) with a 
single continuum-valued upper semi-continuous idempotent bonding function. 
They proved that when \(f\) is a surjection besides the identity, 
the graph of \(f\) must contain two distinct points 
\(\tuple{x,x},\tuple{y,y}\) on the diagonal, in addition
to a third point \(\tuple{x,y}\). This Condition \(\Gamma\)
is sufficient to guarantee that whenever
the inverse limit is indexed by an ordinal \(\alpha\),
the inverse limit contains a copy of \(\alpha+1\). It 
therefore follows that such an inverse limit is a metric continuum if and only if 
\(\alpha\) is countable. The authors suspected, however, that this was merely
a special case of a more general trend. To this end, it will be shown
that to guarantee Condition \(\Gamma\), \([0,1]\) may be replaced with any
weakly countably compact space, \(f\) need not be continuum-valued,
and the surjectivity of \(f\) may be replaced with a weaker assumption
to prevent trivialities. Applying this result, we also prove metrizability theorems for certain Mahavier products indexed by ordinals.

\section{Definitions and Conventions}

Except when otherwise stated,
we assume that all topological spaces are Hausdorff. By convention, all natural numbers are ordinals, e.g. \(4=\{0,1,2,3\}\).
Let \(B^A\) denote the set of functions from \(A\) to \(B\);
in particular, \(X^2\) is the usual square of ordered pairs 
\(\{\tuple{x,x}:x\in X\}\) (each pair is a function from \(\{0,1\}\) to \(X\))
and \(2^X\) is the Cantor set of functions from \(X\) to \(\{0,1\}\).
Let \(F(X)\) denote the non-empty closed subsets of \(X\).

Given a relation \(R\subseteq X\times Y\), let 
\(R(x)=\{y:xRy\}=\{y:\tuple{x,y}\in R\}\); we often will use letters
\(f,g\) to define relations and treat them as set-valued functions in this way.
Given relations or set-valued functions \(f,g\) and a set \(A\), let
\(f(A)=\{f(x):x\in A\}\) and
\(
  (f\circ g)(x)=f(g(x))=\bigcup_{y\in g(x)}f(y)
=
  \{z:\exists y(z\in f(y),y\in g(x))\}
\).
If \(f:X\to F(X)\) or \(f\subseteq X^2\), let \(f^2=f\circ f\).

\begin{definition}
Suppose \(\tuple{P,\leq}\) is a directed set (\(\leq\) is transitive and reflexive, and each pair of points shares a common upper bound)
and for each \(p \in P\), \(X_p\) is a space. 
Let \(\Pi=\prod_{p\in P}X_p\); we will use boldface letters such as 
\(\vect{x}\in\Pi\) to denote sequences in \(\Pi\).
Suppose further that for each \(p \leq q \in P\), there is a set-valued 
\term{bonding function}
\(f_{p,q} : X_q \to F(X_p)\) 
such that for \(p\leq q\leq r\),
\(
  f_{p,r}=f_{p,q}\circ f_{q,r}
\) 
with \(f_{p,p}:X_p\to F(X_p)\) 
defined by \(f_{p,p}(x)=\{x\}\).
Then the \term{generalized inverse limit} 
\(\genInvLim{X_p, f_{p,q}, P}\subseteq\Pi\) is given by:

\[
  \genInvLim{X_p, f_{p,q}, P} 
= 
  \{
    \textbf{x}\in \Pi 
  : 
    p \leq q \Rightarrow \vect{x}(p) \in f_{p,q}(\vect{x}(q)) 
  \}
.\]
\end{definition}

The preceding definition is based upon the one given in \cite{IngMahbook}.
Much of the literature assumes that \(P\) is a total order,
often simply \(\omega=\mathbb{N}=\{0,1,2,\dots\}\), but
allowing for other orders such as \(\mathbb{Z}\) enables
the construction of many interesting examples unattainable with
simply \(\omega\) \cite{vernon}.

The study of generalized inverse limits typically focuses on
upper semi-continuous bonding functions \(f_{p,q}:X_q\to F(X_p)\), 
which for compact
spaces may be characterized as those that map points to non-empty sets
and whose graphs are closed in \(X_q\times X_p\). 
In fact, it will be convenient to simply consider
this graph itself as a relation. This is
exactly the approach taken with Mahavier products
in \cite{charroe}, \cite{GreenKen}.

\begin{definition}
Take the assumptions of the previous definition,
but let \(\tuple{P,\leq}\) be a preordered set (\(\leq\) is transitive and reflexive) and let \(f_{p,q}\subseteq X_q\times X_p\)
such that \(f_{p,r}\subseteq f_{p,q}\circ f_{q,r}\).
Then the \term{Mahavier product} 
\(\maProd{X_p, f_{p,q}, P}\subseteq\Pi\) is given by:

\[
  \maProd{X_p, f_{p,q}, P} 
= 
  \{
    \textbf{x}\in \Pi 
  : 
    p \leq q \Rightarrow \vect{x}(p) \in f_{p,q}(\vect{x}(q)) 
  \}
\]
\end{definition}

Note that the condition 
\(f_{p,r}\subseteq f_{p,q} \circ f_{q,r}\)
has been weakened from equality,
\(f_{p,q}\) is an arbitrary relation from
\(X_q\) to \(X_p\),
and the directed set is now only preordered,
as is done in \cite{charroe}.

\begin{definition}
A Mahavier product \(\maProd{X_p,f_{p,q},P}\) is \term{exact}
whenever \(f_{p,r}=f_{p,q}\circ f_{q,r}\) for all \(p\leq q\leq r\).
\end{definition}

\begin{definition}
A relation \(f\subseteq X\times Y\) is said to be \term{serial}
(also called \term{full} \cite{charroe} or \term{left-total})
if  \(\forall x \in X\) \(\exists y \in Y\) such that 
\(\tuple{x,y}\in f\); that is, \(f(x)\not=\emptyset\)
for all \(x\in X\).
\end{definition}

Note that we will refrain from using the term ``full'' to describe such 
relations, as ``full'' bonding functions are defined in another sense in 
\cite{GreenKen2}.

\begin{observation}
Exact Mahavier products bonded by 
closed-valued serial relations
indexed by a directed set
are generalized inverse limits.
\end{observation}

\begin{definition}
A serial relation that is a closed subset of \(X \times Y\) 
is called \term{USC}.
\end{definition}

Here USC stands for upper semi-continuous, as
in the case that \(X,Y\) are compact, this property may be characterized
as follows: for every \(x\in X\) and open set \(V\supseteq f(x)\),
there exists an open neighborhood \(U\) of \(x\) such that
\(f(u)\subseteq V\) for all \(u\in U\) \cite{IngMah1}.
Put another way, the set-valued function \(f:X\to F(Y)\) is continuous
where \(F(Y)\) is given the upper Vietoris topology.

\begin{definition}
A relation \(f\subseteq X\times Y\) is \term{surjective} if, 
for all \(y\in Y\), there exists \(x \in X\) such that 
\(\tuple{x,y}\in f\).
\end{definition}

\begin{definition}
An \term{idempotent} relation on \(f\subseteq X^2\) 
is one that satisfies \(f^2=f\).
\end{definition}

Note that transitivity may be characterized by
\(f^2\subseteq f\), so all idempotent relations are transitive.
Put another way, an idempotent relation \(R\subseteq X^2\)
satisfies \(xRz\) if and only if \(xRyRz\) for some \(y\in X\);
transitivity is the backwards implication.
Assuming reflexivity, idempotence and transitivity are equivalent:
let \(y=z\), so \(xRz\Leftrightarrow xRzRz \Leftrightarrow xRyRz\).
Note also that the inverse of an idempotent relation is idempotent.

The usual strict linear order \(<\) on \(\mathbb{Z}\) is
an example of a transitive yet non-idempotent relation.
Likewise, the strict linear order on any dense subset of \(\mathbb R\)
is a non-reflexive idempotent relation.

Idempotent relations are of significant interest
when studying exact Mahavier products of many copies of the same topological space.

\begin{observation}
If \(X_p=X\) and \(f_{p,q}=f\) for all \(p<q\),
and there exist \(p,q,r\in P\) such that \(p<q<r\), 
then the bonding relation \(f\) 
in an exact Mahavier product \(\maProd{X,f,P}\) must be idempotent. 
\end{observation}

Note that \(f_{p,p}\not=f\) unless \(f\) is the identity, but for
simplicity we still simply write \(\maProd{X,f,P}\).

\begin{definition}
For convenience, we call a USC idempotent surjective relation 
\(f \subseteq X^2\) a \term{V-relation}. 
\end{definition}

\begin{figure}
\begin{center}
\begin{tikzpicture}
\begin{scope}[shift={(0,0)}]
  \draw[gray] (0,0) rectangle (2,2);
  \draw[thick,blue] (0,0) -- (2,2);
\end{scope}
\begin{scope}[shift={(3,0)}]
  \draw[gray] (0,0) rectangle (2,2);
  \draw[thick,blue] (0,2) -- (0,0) -- (2,2);
  \draw[black,fill=red] (0,0) circle (0.1);
  \draw[black,fill=red] (0,1) circle (0.1);
  \draw[black,fill=red] (1,1) circle (0.1);
\end{scope}
\begin{scope}[shift={(6,0)}]
  \draw[gray] (0,0) rectangle (2,2);
  \draw[thick,blue] (0,0) -- (0.5,0.5) -- (0.5,2) -- (2,2) -- (1.5,1.5)
    -- (1.5,2); 
  \draw[black,fill=red] (0.5,0.5) circle (0.1);
  \draw[black,fill=red] (0.5,1.5) circle (0.1);
  \draw[black,fill=red] (1.5,1.5) circle (0.1);
\end{scope}
\begin{scope}[shift={(9,0)}]
  \draw[gray] (0,0) rectangle (2,2);
  \draw[thick,blue] (0,0) -- (1,1) -- (2,1);
  \draw[thick,blue,fill] (1.5,1) -- (1.5,2) -- (2,2) -- (2,1);
  \draw[black,fill=red] (1.5,1.5) circle (0.1);
  \draw[black,fill=red] (2,1.5) circle (0.1);
  \draw[black,fill=red] (2,2) circle (0.1);
\end{scope}
\end{tikzpicture}
\end{center}
\caption{Illustrating Condition \(\Gamma\) for V-relations on \([0,1]\)}
\label{VRelation}
\end{figure}
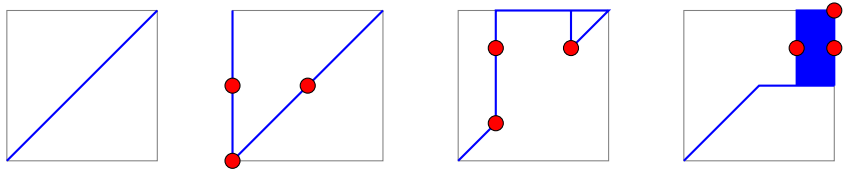

Examples of V-relations are given in Figure \ref{VRelation}
(the designated points illustrate the definition of Condition \(\Gamma\),
given in Section \ref{conditionGamma}).
The next section will outline how to construct and identify
V-relations.

\section{Constructing V-relations}

The following propositions not only give necessary and/or sufficient conditions 
for a relation \(f\) to be a V-relation, but also give the reader a few tools for 
constructing simple V-relations from scratch. The first proposition is a
useful recharacterization of idempotence.

\begin{proposition}\label{Vrelcondition}
Let $X$ be a space and let $f \subseteq X^2$ be USC and surjective. Then $f$ is a V-relation iff whenever $A$ is the image of some $x \in X$, then $f(A)=A$.
\end{proposition}

As an obvious consequence of Proposition \ref{Vrelcondition}, if $f \subseteq X^2$ 
is a V-relation and for some $x \in X$, $f(x)=\{y\}$, then $f(y)=\{y\}$.

\begin{proposition} \label{Vrelcondition2}
Let $f \subseteq [0,1]^2$ be USC. If $f$ satisfies at least one of the following conditions, then $f$ is a V-relation.

\begin{enumerate}
\item For each $x \in [0,1]$, $f(x)=\{x, 1-x\}$ .

\item The relation $f$ is surjective, and for each $x \in [0,1]$, $f(x)=[0,x]$ 
or $f(x) \subseteq \{0, x\}$.

\item
The relation $f$ is surjective, and for each $x \in [0,1]$, $f(x)=[x,1]$ or 
$f(x) \subseteq \{x,1\}$.

\item
For some non-empty $A, B \subseteq [0,1]$ with $A \cap B = \emptyset$, we have 
$f(a)=[0,1]$ for each $a \in A$ and $f(x)=B$ for each $x \in [0,1] \setminus A$.

\item  
The relation $f$ is surjective, and there exists $b\in [0,1]$ with 
$b\in f(b)=B$ so that for all $x \in [0,1]$, 
either $f(x)=\{x\}$, $f(x)=B$, or $f(x)=\{y\}$ for some $y\in B$ 
satisfying $f(y)=\{y\}$.
\end{enumerate}
\end{proposition}

\begin{proof}
Each of the five conditions implies that $f$ is surjective, so it remains to verify that each condition implies that $f$ is idempotent.  Ingram already observed that condition 1 implies $f$ is idempotent in \cite{Ingrambook}. For condition 2, let $x \in [0,1]$. We note that $f(y)\subseteq [0,x]$ for every $y \le x$. Therefore, if $f(x)= [0,x]$, then clearly $f^2(x)= [0,x]=f(x)$. On the other hand, if $f(x) = \{0,x\}$, then $f^2(x)=f(0) \cup f(x) = \{0\} \cup \{0,x\}=f(x)$; the remaining cases are obvious. The details for condition 3 are similarly straightforward and are left to the reader.

For condition 4, we note that when $x \in A$, $f(x)=[0,1]$ (so of course $f^2(x)=[0,1]$), whereas if $x \in [0,1] \setminus A$, then $f(x)=B$, so that $f^2(x)=f(B)$. However, when $b \in B$, since $b \not\in A$, it follows that $f(b)=B$; thus, $f^2(x)=B$, and we conclude that $f$ is idempotent.

To prove that condition 5 implies $f$ is idempotent, let $x \in [0,1]$. If $f(x)=\{x\}$ then clearly $f^2(x)=f(x)$. If $f(x)=\{y\}$ for some $y \in B$ with $f(y)=\{y\}$, then $f^2(x)=f(y)=\{y\}=f(x)$. Finally, if $f(x)=B$, then $f^2(x)=f(B)$; since $f(b) \subseteq B$ for each $b \in B$, and there is some $b \in B$ that satisfies $f(b)=B$, it follows that $f(B)=B$. Thus, $f^2(x)=f(x)$ in each case.
\end{proof}

Note that Proposition \ref{Vrelcondition2} implies that each of the relations pictured in Figure \ref{VRelation}  is a V-relation. (Of course, the sufficient conditions given in Proposition \ref{Vrelcondition2} are by no means exhaustive.)

\section{
Condition \texorpdfstring{\(\Gamma\)}{Gamma}
}\label{conditionGamma}

\begin{definition}
  A relation \(f \subseteq X^2\) satisfies \term{Condition \(\Gamma\)} if there 
  exist distinct \(x, y \in X\) such that \(\tuple{x,x},\tuple{x,y},\tuple{y,y}\in f\).
\end{definition}

Note that in Figure \ref{VRelation}, every example given besides the
identity satisfies Condition \(\Gamma\).
In fact, this section will show that every V-relation besides the identity on a 
weakly countably compact space satisfies Condition \(\Gamma\).

\begin{definition}
  Let \(\iota\subseteq X^2\) be the diagonal
  \(\iota=\{\tuple{x,x}:x\in X\}\), i.e. the identity relation.
\end{definition}

\begin{definition}
  For \(f\subseteq X\times Y\) and \(A\subseteq X\), 
  let \(f\rest A=f\cap (A\times Y)\). 
\end{definition}

\begin{proposition}
  For \(f\) transitive, \(f\rest f(x)=f\cap (f(x))^2\).
\end{proposition}

\begin{definition}
  A relation is said to be \term{trivial} if for all
  \(x\in X\), \(f\rest f(x)=\iota \rest f(x)\).
\end{definition}

Any trivial relation is idempotent, and
of course \(\iota\) is trivial, but there exist trivial
idempotent relations besides the identity. For example, let
\(
  t\subseteq 3^2=\{0,1,2\}^2
\)
be defined by
\[
  t=\{\tuple{0,0},\tuple{1,1},\tuple{2,0},\tuple{2,1}\}
.\]
It follows that \(t\rest f(0)=t\rest\{0\}=\iota\rest\{0\}\),
\(t\rest f(1)=t\rest\{1\}=\iota\rest\{1\}\),
and \(t\rest f(2)=t\rest\{0,1\}=\iota\rest\{0,1\}\).
Figure \ref{trivialRelation} shows examples of trivial relations
defined on \([0,1]\); note that these do not satisfy Condition \(\Gamma\).

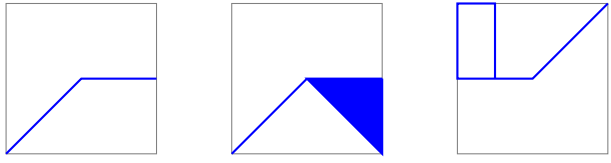
\begin{figure}
\begin{center}
\begin{tikzpicture}
\begin{scope}
  \draw[gray] (0,0) rectangle (2,2);
  \draw[thick,blue] (0,0) -- (1,1) -- (2,1);
\end{scope}
\begin{scope}[shift={(3,0)}]
  \draw[gray] (0,0) rectangle (2,2);
  \draw[thick,blue] (0,0) -- (1,1);
  \draw[thick,blue,fill=blue] (2,1) -- (2,0) -- (1,1) -- (2,1);
\end{scope}
\begin{scope}[shift={(6,0)}]
  \draw[gray] (0,0) rectangle (2,2);
  \draw[thick,blue] (2,2) -- (1,1) -- (0,1) -- (0,2) -- (0.5,2) -- (0.5,1);
\end{scope}
\end{tikzpicture}
\end{center}
\caption{Trivial idempotent relations on \([0,1]\) besides \(\iota\)}
\label{trivialRelation}
\end{figure}

Observing that none of the examples in Figure \ref{trivialRelation}
are surjective, the following proposition will allow us to ignore
such exceptional cases when considering V-relations besides the identity. 

\begin{proposition}
  Let \(f\) be an idempotent surjective relation on \(X\).
  Then the following are equivalent.
  \begin{enumerate}[a)]
  \item \(f=\iota\)
  \item \(|f(x)|=1\) for all \(x\in X\)
  \item \(f\) is trivial
  \end{enumerate}
\end{proposition}
\begin{proof}
  (a) implies (b) trivially, so assume (b). For an arbitrary \(x\in X\),
  there exists some \(y\in X\) such that \(f(x)=\{y\}\). 
  Thus \(f^2(x)=\{y\}=f(y)\) by transitivity, so it follows that 
  \(f\rest\{y\}=\{\tuple{y,y}\}=\iota\rest\{y\}\), showing (c).
  
  Now assuming (c), take \(y\in X\). As \(f\) is surjective, 
  there exists \(x\in X\) where \(y\in f(x)\). By idempotence,
  \(y\in f^2(x)\), and in particular, \(y\in f(y)\).
  But since \(f\rest f(y)=\iota\rest f(y)\) by triviality,
  we note that \(f(y)=(f\rest f(y))(y)=(\iota\rest f(y))(y)=\iota(y)=\{y\}\),
  showing (a).
\end{proof}

\begin{corollary}
  \(\iota\) is the only trivial V-relation.
\end{corollary}

We now aim to exploit the properties of non-trivial relations
to produce Condition \(\Gamma\) in V-relations besides \(\iota\).
To do this, we require the following lemmas.

\begin{definition}
A weakly countably compact space is a space such that every infinite
subset has a limit point.
\end{definition}

Assuming spaces are \(T_2\), this property is equivalent
to countable compactness: every countable open cover has a finite subcover.
But we will need only this weaker characterization, and the remainder
of this section does not assume any separation axioms for the spaces under
consideration.

\begin{lemma}\label{lemma-usc-idempotent}
  Let \(X\) be weakly countably compact.
  Every USC idempotent relation \(f\subseteq X^2\) besides \(\iota\)
  contains two points \(\tuple{y,y},\tuple{x,y}\) for some 
  distinct \(x,y\in X\).
\end{lemma}
\begin{proof}
  Note first that if \(\iota\subseteq f\), then since
  \(\iota\not=f\) the lemma follows
  immediately. So 
  let \(x_0\in X\) be a point where \(\tuple{x_0,x_0}\not\in f\).
  
  Suppose \(x_i\) is defined for \(i\leq n\) such that
  \(\tuple{x_i,x_j}\in f\) if and only if \(i<j\).
  Since \(\{x_0,\dots,x_n\}\cap f(x_n)=\emptyset\),
  we may choose \(x_{n+1}\not\in\{x_0,\dots,x_n\}\)
  such that \(\tuple{x_n,x_{n+1}}\in f\).
  Note that by idempotence, \(\tuple{x_{n+1},x_i}\not\in f\)
  for \(i\leq n\) since 
  \(x_i\not\in f(x_n)=f^2(x_n)\supseteq f(x_{n+1})\).
  Similarly, \(\tuple{x_i,x_{n+1}}\in f\) for \(i<n\) since
  \(\tuple{x_i,x_n}\in f\) and \(\tuple{x_n,x_{n+1}}\in f\).
  If \(\tuple{x_{n+1},x_{n+1}}\in f\),
  then the lemma is satisfied by \(x=x_n\) and \(y=x_{n+1}\).
  
  If not, we have recursively constructed an infinite
  set \(\{x_n:n<\omega\}\).
  Since \(X\) is a weakly countably compact space, 
  \(\{x_n:n<\omega\}\) has a limit point \(x_\omega\).
  Since \(\{\tuple{x_n,x_{n+1}}:n<\omega\}\subseteq f\), it follows that
  \(\tuple{x_\omega,x_\omega}\) is a limit point of \(f\).
  Similarly, \(\{\tuple{x_0,x_n}:0<n<\omega\}\subseteq f\),
  so it follows that \(\tuple{x_0,x_\omega}\)
  is also a limit point of \(f\). Since \(f\) is closed,
  these limit points belong to \(f\).
  Therefore the lemma is witnessed by
  \(x=x_0\) and \(y=x_\omega\).
\end{proof}

\begin{lemma}
  Let \(f\subseteq X^2\) be idempotent and serial, and let \(x_0\in X\).
  Then \(f\rest f(x_0)\) and its inverse are idempotent, serial,
  surjective relations on \(f(x_0)\).
\end{lemma}
\begin{proof}
  Let \(g=f\rest f(x_0)\). For each \(x\in f(x_0)\), \(f(x)\subseteq f^2(x_0)=f(x_0)\).
  Thus \(g(x)=f(x)\not=\emptyset\) and 
  \(g^2(x)=f^2(x)=f(x)=g(x)\), showing that \(g\) is idempotent and serial.
  It is also surjective: for \(y\in f(x_0)\), \(y\in f^2(x_0)\), so there
  exists some \(x\in f(x_0)\) such that \(g(x)=f(x)=y\).
  Since \(g\) is serial, surjective, and idempotent, so is \(g^{-1}\).
\end{proof}

\begin{lemma}\label{lemma-nontrivial-idempotent}
  Let \(f\subseteq X^2\) be idempotent and serial, and let \(x_0\in X\)
  witness that \(f\) is non-trivial.
  Then \(f\rest f(x_0)\) and its inverse are idempotent, serial,
  surjective, non-trivial relations on \(f(x_0)\).
\end{lemma}
\begin{proof}
  By the previous lemma, \(g=f\rest f(x_0)\) and \(g^{-1}\) 
  are idempotent, serial, and surjective.
  By non-triviality, \(g\not=\iota \rest f(x_0)\),
  so \(g^{-1}\not=\iota\rest f(x_0)\) too.
  But since \(\iota\rest f(x_0)\) 
  is the only surjective idempotent trivial relation on \(f(x_0)\),
  neither \(g\) nor \(g^{-1}\) are trivial.
\end{proof}

\begin{theorem}
  Let \(X\) be weakly countably compact, and let \(f\subseteq X^2\) be
  an idempotent USC relation. Then the following are equivalent.
  \begin{enumerate}[a)]
  \item \(f\) satisfies Condition \(\Gamma\)
  \item \(f\) contains points \(\tuple{x,x},\tuple{x,y}\) for some distinct
        \(x,y\in X\)
  \item \(f\) is non-trivial
  \end{enumerate}
\end{theorem}

\begin{proof}
  (a) implies (b) trivially. Assuming (b), it follows that
  \(f\rest f(x)\not=\iota\rest f(x)\) since \(\tuple{x,y}\in f\rest f(x)\),
  showing (c).
  
  Assuming (c), we may apply Lemma \ref{lemma-nontrivial-idempotent} 
  to choose \(x_0\)
  such that \(g=f\rest f(x_0)\) and \(g^{-1}\) are idempotent, USC,
  surjective, non-trivial relations. 
  Then Lemma \ref{lemma-usc-idempotent} allows us to conclude that
  \(g^{-1}\) contains points \(\tuple{x,x},\tuple{y,x}\) for some distinct
  \(x,y\in X\), so (b) is satisfied.
  
  Assume (b) such that \(\tuple{y,y}\not\in f\).
  Note then that \(\tuple{y,x}\not\in f\) as otherwise
  \(\tuple{y,x},\tuple{x,y}\in f\) would imply \(\tuple{y,y}\in f\).
  Let \(z_0=x\) and \(z_1=y\). 
  
  Suppose \(z_i\) is defined for \(i\leq n+1\) such that
  \(\tuple{z_i,z_j}\in f\) if and only if \(i<j\) or \(i=j=0\).
  Since \(\{z_0,\dots,z_{n+1}\}\cap f(z_{n+1})=\emptyset\),
  we may choose \(z_{n+2}\) distinct from \(z_i\) for \(i\leq n+1\) 
  such that \(\tuple{z_{n+1},z_{n+2}}\in f\).
  Note that by idempotence, \(\tuple{z_{n+2},z_i}\not\in f\)
  for \(i\leq n+1\) since 
  \(z_i\not\in f(z_{n+1})=f^2(z_{n+1})\supseteq f(z_{n+2})\).
  Similarly, \(\tuple{z_i,z_{n+2}}\in f\) for \(i<n+1\) since
  \(\tuple{z_i,z_{n+1}}\in f\) and \(\tuple{z_{n+1},z_{n+2}}\in f\).
  If \(\tuple{z_{n+2},z_{n+2}}\in f\),
  then Condition \(\Gamma\) is witnessed by 
  \(\tuple{z_0,z_0},\tuple{z_0,z_{n+2}},\tuple{z_{n+2},z_{n+2}}\).
  
  If not, we have recursively constructed an infinite
  set \(\{z_n:n<\omega\}\).
  Since \(X\) is a weakly countably compact space, 
  \(\{z_n:n<\omega\}\) has a limit point \(z_\omega\).
  Since \(\{\tuple{z_n,z_{n+1}}:n<\omega\}\subseteq f\), it follows that
  \(\tuple{z_\omega,z_\omega}\) is a limit point of \(f\).
  Similarly, \(\{\tuple{z_0,z_n}:n<\omega\}\subseteq f\),
  so it follows that \(\tuple{z_0,z_\omega}\)
  is also a limit point of \(f\). Since \(f\) is closed,
  these limit points belong to \(f\).
  Therefore Condition \(\Gamma\) is witnessed by
  \(\tuple{z_0,z_0},\tuple{z_0,z_\omega},\tuple{z_\omega,z_\omega}\),
  showing \((a)\).
\end{proof}

\begin{corollary}Let $X$ be weakly countably compact, and let $f \ne \iota$ be a V-relation on $X$. Then $f$ satisfies Condition $\Gamma$.
\end{corollary}

It is worth noting that the strict linear order \(<\) on \(\mathbb Q\) 
with the discrete topology is an example of a non-trivial idempotent
USC relation on a space that is not weakly countably compact
that does not satisfy Condition \(\Gamma\). Likewise,
the strict lexicographic order \(<\) on the long ray
\(\omega_1\times[0,1]\) with the topology induced by this linear order 
is an example of a non-trivial idempotent
serial non-USC relation on a Hausdorff countably compact 
space that does not satisfy Condition \(\Gamma\).

\section{Applications}

Let \(\alpha=\{\beta:\beta<\alpha\}\) be an ordinal with its usual 
linear order. As noted in \cite{CLONTZVARAGONA}, it is well-known and 
easy to see that \(\alpha+1\) as a totally ordered topological space
is metrizable if and only if \(\alpha\) is countable: if 
\(\alpha\) is uncountable, then the first uncountable ordinal
\(\omega_1\) is
a point of non-first-countability in \(\alpha+1\); if \(\alpha\)
is countable, then \(\alpha+1\) is regular and second-countable.

\begin{theorem}
  Let \(X\) be a \(T_1\) topological space, let
  \(\alpha\) be an uncountable ordinal, and let \(f\) be a relation on $X$ (with $f \subseteq f \circ f$)
  satisfying Condition \(\Gamma\). 
  Then the Mahavier product 
  \(\maProd{X,f,\alpha}\) contains a copy of \(\alpha+1\); 
  therefore, \(\maProd{X,f,\alpha}\) cannot be metrizable.
\end{theorem}

\begin{proof}
  Let \(\tuple{x,x},\tuple{x,y},\tuple{y,y}\in f\)
  for distinct \(x,y\in X\). For \(\gamma\leq \alpha\),
  define \(\vect{x}_\gamma\in\maProd{X,f,\alpha}\) by
  \(\vect{x}_\gamma(\beta)=y\) for \(\beta<\gamma\)
  and \(\vect{x}_\gamma(\beta)=x\) for \(\gamma\leq\beta<\alpha\).
  That is, \(\vect{x}_\gamma\) is defined such that \(\gamma\) is
  the least ordinal such that \(\vect{x}_\gamma(\gamma)=x\).

  We will show that the map \(\gamma\mapsto\vect{x}_\gamma\) is
  a homeomorphism from \(\alpha+1\) to
  \(A=\{\vect{x}_\gamma:\gamma\leq\alpha\}\subseteq\maProd{X,f,\alpha}\). 
  This may be accomplished by comparing subbases: for \(\beta\leq\alpha\),
  the subbasic open set \([0,\beta)\subseteq\alpha+1\) maps to
  the set \(\{\vect{x}\in A:\vect{x}(\beta)=x\}\subseteq A\),
  which equals the open set \(A\cap\prod_{\gamma<\alpha}U_\gamma\) where
  \(U_\beta=X\setminus\{y\}\) and \(U_\gamma=X\) otherwise.
  Similarly, the subbasic open set \((\beta,\alpha]\subseteq\alpha+1\) 
  maps to the set \(\{\vect{x}\in A:\vect{x}(\beta)=y\}\subseteq A\), 
  which equals the open set \(A\cap\prod_{\gamma<\alpha}U_\gamma\) where
  \(U_{\beta}=X\setminus\{x\}\) and \(U_\gamma=X\) otherwise.
  And since \(A\cap\prod_{\gamma<\alpha}U_\gamma\) describes every
  subbasic open subset of \(A\) for 
  \(U_\beta\in\{\emptyset,X\setminus\{x\},X\setminus\{y\},X\}\)
  and \(U_\gamma=X\) otherwise,
  it follows that this map is indeed a homeomorphism.
\end{proof}

\begin{corollary}
  Let \(X\) be any weakly countably compact metrizable space, let
  \(\alpha\) be an ordinal, and let \(f\not=\iota\) be a 
  V-relation on $X$. Then the exact Mahavier product 
  \(\maProd{X,f,\alpha}\) contains a copy of \(\alpha+1\); 
  therefore, \(\maProd{X,f,\alpha}\) is metrizable
  if and only if \(\alpha\) is countable.
\end{corollary}

As noted in \cite{CLONTZVARAGONA}, a bit more can be said.

\begin{definition}
  The \term{\(\Sigma\)-product of reals \(\Sigma\mathbb{R}^\kappa\)}
  for a cardinal \(\kappa\) is given by
  \[\{
    \vect x\in \mathbb{R}^\kappa
  :
    |\{\alpha<\kappa:\vect{x}(\alpha)\not=0\}|\leq\aleph_0
  \}\]
\end{definition}

Note \(\Sigma\mathbb{R}^\omega=\mathbb{R}^\omega\), and
since every compact metrizable space embeds in \([0,1]^\omega\),
it follows that every compact metrizable space embeds in
a \(\Sigma\)-product of reals.

Compact subspaces of \(\Sigma\mathbb{R}^\kappa\) are known
as Corson compacts: see e.g. \cite{Corson} for
an investigation into the applications of Corson compacts in functional
analysis.

\begin{corollary}
  Let \(X\) be a \(T_1\) topological space, let
  \(\alpha\) be an uncountable ordinal, and let \(f\) be a relation on $X$ (with $f \subseteq f \circ f$)
  satisfying Condition \(\Gamma\). 
  Then the Mahavier product 
  \(\maProd{X,f,\alpha}\) cannot be embedded in a \(\Sigma\)-product
  of reals.
\end{corollary}
\begin{proof}
The Mahavier product \(\maProd{X,f,\alpha}\) contains a copy of \(\alpha+1\),
  which cannot be embedded into a \(\Sigma\)-product of reals.
\end{proof}

\section{Acknowledgements}

The authors wish to thank Jonathan Meddaugh, Christopher Night,
and the participants in the Spring 2016 Carolina
Topology Seminar for their comments on various working versions
of these results.

\end{document}